\documentclass[12pt]{amsart}
\usepackage{amsfonts}
\usepackage{amssymb}
\usepackage{color}
\usepackage{dsfont}
\usepackage{MnSymbol}
\usepackage{amsmath}
\usepackage{amsthm}
\usepackage{mathrsfs}

\numberwithin{equation}{section}
\theoremstyle{plain}
 \newtheorem{theorem}{Theorem}[section]
 
 \newtheorem{corollary}[theorem]{Corollary}
 \newtheorem{proposition}[theorem]{Proposition}

\theoremstyle{definition}
 \newtheorem{definition}[theorem]{Definition}
 \newtheorem{example}[theorem]{Example}
 \newtheorem{remark}[theorem]{Remark}

\newcommand{\C}{\mathbb{C}}
\newcommand{\R}{\mathbb{R}}
\newcommand{\N}{\mathbb{N}}

\newcommand{\bN}{\mathbb{N}}

\newcommand{\bR}{\mathbb{R}}

\newcommand{\supp}{\mbox{\rm supp }}

\newcommand{\bE}{\mathbb{E}}

\allowdisplaybreaks

\begin{document}

\vspace{5mm}
\begin{center}
{\bf
{\large
Second order elliptic partial differential equations driven by L\'{e}vy white noise}}

\vspace{5mm}

David Berger and Farid Mohamed\\
\end{center}
\vspace{1cm}

This paper deals with linear stochastic partial differential equations with variable coefficients driven by L\'{e}vy white noise. We first derive an existence theorem for integral transforms of L\'{e}vy white noise and prove the existence of generalized and mild solutions of second order elliptic partial differential equations. Furthermore, we discuss the generalized electric Schr\"odinger operator for different potential functions $V$.

\section{Introduction}
Since the beginning of studying partial differential equations the Laplacian operator $\Delta:=\sum\limits_{j=1}^d \partial_{j}^2$ was of great interest in different mathematical theories and applications. For example, the solution of the Poisson equation
\begin{align*}
    -\Delta u=f
\end{align*}
for some function $f$ can be interpreted as a stationary solution of the heat equation and is therefore important in thermodynamics. In order to study different heterogeneity assumptions in the space, the divergence operator
\begin{align*}
    \textrm{div}(A(x)\nabla u):=\sum\limits_{i,j=1}^d \partial_i (a_{ij}(x) \partial_j u)
\end{align*}
was introduced, where the matrix function $A$ satisfies some ellipticity condition. This kind of operator is for example used in the Maxwell equations in general media (see [\ref{Putten}]).\\
The fundamental solution of the Laplace equation is well-known, but there is no explicit form for a fundamental solution of a general divergence form operator, but there exist upper and lower bounds, see for example [\ref{Littman}]. \\
The goal of this paper is to obtain generalized solutions of the equation
\begin{align*}
    p(x,D) s =  \dot{L},
\end{align*}
where $\dot{L}$ is a so-called generalized L\'{e}vy white noise and $p$ is a partial differential operator of the form 
\begin{align}
    -\textrm{div} (A(x)\nabla u)+b(x)\cdot \nabla u+V(x)u, \quad u\in C^{\infty}(\R^d)\label{pdov},
\end{align}
for a uniformly elliptic $\R^d$-valued matrix function $A$ and functions $b:\R^d\to\R^d$, $V:\R^d\to\R$. We especially achieve generalized and mild solutions for the generalized electric Schr\"odinger operator driven by a L\'evy white noise, i.e. we are looking for a solution $u$ of the stochastic partial differential equation 
\begin{align}
    -\textrm{div}(A(x)\nabla u) + V(x)u = \dot{L},\label{Schrod2}
\end{align}
where $A$ is a uniformly elliptic $d\times d$ matrix, the potential $V> 0$ belongs to the reverse H\"older class and $\dot{L}$ is a L\'evy white noise. Since the fundamental solution of the Schr\"odinger operator has exponential decay, we will derive weaker assumptions on the L\'evy white noise in comparison to the general case (\ref{pdov}) to show the existence of generalized and mild solutions. This can be seen as an extension of the theory founded in [\ref{Berger}] by D. Berger, but the results are not directly applicable. In order to overcome this shortcoming we derive existence results for generalized random processes constructed by integral transforms of the underlying L\'{e}vy white noise. Furthermore, we study different distributional properties of these solutions and show that we can construct periodically stationary generalized random processes.\\
We are solving the stochastic partial differential equations in distributional sense, i.e. a solution $s$ is a distribution valued random variable such that $\langle s, p(x,D)^*\varphi\rangle =\langle \dot{L},\varphi\rangle$ for every $\varphi$ in our function space. For a good introduction to distributional solutions of partial differential equations see for example [\ref{Hoermander}]. Until now there does not exist a good understanding of L\'{e}vy white noise driven stochastic partial differential equations under general moment conditions, but there exists literature for the case of Gaussian white noise and L\'{e}vy white noise with stricter moment conditions. In [\ref{Walsh}]  SPDEs driven by Gaussian white noise where studied. Moreover, a similar approach for L\'{e}vy white noise can be found in [\ref{Holden}] and [\ref{Lokka}]. In the case of stochastic partial differential equations with constant coefficients see also [\ref{Dalang}] and [\ref{Berger}]. Our method is inspired by the papers of [\ref{Fageot}] and the results of [\ref{Rajput}].

In Section 3 we provide the general framework needed to discuss stochastic partial differential equations driven by L\'evy white noise, whose solutions are defined as generalized random process. We introduce L\'evy white noise as a generalized random process in the sense of I.M. Gelfand and N.Y. Vilenkin (see [\ref{Gelfand}]). Theorem \ref{theo1} implies that a large class of linear stochastic partial differential equations driven by a L\'evy white noise has a generalized solution, where we used a more general kernel $G:\R^m\times \R^d\to\R$ compared to Theorem 3.4 of D. Berger in [\ref{Berger}]. Furthermore, we study the moment properties of generalized random processes $s$ driven by L\'evy white noise $\dot{L}$. For a well-defined random process $s(\varphi) = \langle \dot{L}, G(\varphi)\rangle$, $\varphi\in\mathcal{D}(\R^d)$ we show in Theorem \ref{momentprop} that if $\dot{L}$ has finite $\beta>0$ moment, then $s$ has also finite $\beta$-moment under further conditions on the kernel $G$. Moreover, we show that if $s$ has finite $\beta$-moment, then also $\dot{L}$ has finite $\beta$-moment. In Section 4 we discuss our first example, the partial differential operators of the form (\ref{pdov}) and give existence results for generalized solutions. Furthermore, we discuss periodically stationary solutions $s$ for this example. Afterwards we consider the generalized electric Schr\"odinger operator driven by L\'evy white noise and show under weaker conditions, as in the example above, the existence of generalized solutions. We also study the concept of mild solutions of (\ref{Schrod2}), i.e. a solution $u$ which is a random field and given by the convolution of the L\'evy white noise with the fundamental solution of (\ref{Schrod2}).
In Proposition \ref{mild} we mention when such a solution $u$ exists and is stochastically continuous.

\section{Notation and Preliminaries}
Let us recall a few key concepts and techniques which will be needed later on: Most of our notation is standard or self-explanatory; where $\lambda^d$ is the Lebesgue measure on $\R^d$.

\section{Integral transforms and generalized stochastic processes driven by Levy white noise}
\noindent We provide the general framework needed to discuss stochastic partial differential equations driven by L\'evy white noise and introduce L\'evy white noise as generalized random processes in the sense of I.M. Gelfand and N.Y. Vilenkin (see [\ref{Gelfand}]). In [\ref{Berger}] it was shown that a convolution operator, with certain properties regarding his integrability, defines a generalized random process, assuming low moment conditions on the L\'evy white noise. Similar to [\ref{Berger}], we will use the characterization of the extended domain (see [\ref{Fageot}, Proposition 3.7.]) and achieve new results for a more general kernel $G:\R^m\times \R^d \to\R$, which allows us in Section 4 to model different kinds of stationarity assumptions and also to obtain generalized solutions of L\'evy driven stochastic partial differential equations.\\
Let $(\Omega,\mathcal{F},\mathcal{P})$ be a probability space.

\begin{definition}(see [\ref{Fageot}, Definition 2.1.])
A \emph{generalized random process} is a linear and continuous function $s:\mathcal{D}(\R^d)\to L^0(\Omega)$. The linearity means that, for every $\varphi_1, \varphi_2 \in \mathcal{D}(\R^d)$ and $\mu\in \R$,
\begin{align*}
s(\varphi_1 + \mu\varphi_2) = s(\varphi_1)+ \mu s(\varphi_2) \textrm{ almost surely.}
\end{align*}
The continuity means that if $\varphi_n\to \varphi$ in $\mathcal{D}(\R^d)$, then $s(\varphi_n)$ converges to $s(\varphi)$ in probability.
\end{definition}
Due to the nuclear structure on $\mathcal{D}(\R^d)$ it follows with [\ref{Walsh}, Corollary 4.2] that a generalized random process has a version which is a measurable function from $(\Omega,\mathcal{F})$ to $(\mathcal{D}'(\R^d),\mathcal{C})$ with respect to the cylindrical $\sigma$-field $\mathcal{C}$ generated by the sets
\begin{align*}
    \{u\in\mathcal{D}'(\R^d)|\,(\langle  u,\varphi_1\rangle,\dotso,\langle u,\varphi_N\rangle)\in B\}
\end{align*}
with $N\in \N$, $\varphi_1,\dotso,\varphi_N \in \mathcal{D}(\R^d)$ and $B\in \mathcal{B}(\R^{N})$.
From now on it is always meant such a version.
\\
The probability law of a generalized random process $s$ is the probability measure on $\mathcal{D}'(\R^d)$ given by
\begin{align*}
\mathcal{P}_s(B):=\mathcal{P}(s\in B) = \mathcal{P}(\{\omega\in \Omega : s(\omega)\in B\})
\end{align*}
for $B\in\mathcal{C}$, where $\mathcal{C}$ is the cylindrical $\sigma$-field on $\mathcal{D}'(\R^d)$. \\
The characteristic functional of a generalized random process $s$ is the functional $\widehat{\mathcal{P}} : \mathcal{D}(\R^d) \to \C$ defined by
\begin{align*}
\widehat{\mathcal{P}}_s(\varphi)=\int\limits_{\mathcal{D}'(\R^d)}\exp(i\langle u,\varphi\rangle)d\mathcal{P}_s(u).
\end{align*}
The characteristic functional characterizes the law of $s$ in the sense that two random processes are equal in law if and only if they have the same characteristic functional.
Now we define the L\'{e}vy white noise, which is closely connected to a L\'{e}vy process. In general, a L\'{e}vy process is a stochastically continuous process with independent and stationary increments starting in $0$. A L\'{e}vy process $(L_t)_{t\ge 0}$ is characterized by its characteristic function, it holds that
\begin{align*}
    \bE e^{izL_t}=\exp(t\psi(z)),
\end{align*}
for every $z\in\R$ and $t\ge 0$. We call $\psi$ the L\'evy exponent which can be characterized by an $a\ge 0$, $\gamma\in \R$ and a L\'{e}vy measure $\nu$, i.e. a measure such that
\begin{align*}
    \nu(\{0\})=0\textrm{ and }\int\limits_{\R\setminus\{0\}}\min\{1,x^2\}\nu(dx)<\infty.
\end{align*}
For all $z\in \R$ it holds that
\begin{align*}
        \psi(z) = i\gamma z- \frac{1}{2} a z^2 + \int\limits_{\R}(e^{i x z}-1-i x z \mathds{1}_{|x|\le 1})\nu(dx).
    \end{align*}
\begin{definition}
A \emph{L\'{e}vy white noise} $\dot{L}$ on $\R^d$ is a generalized random process with characteristic functional of the form
\begin{align*}
\widehat{\mathcal{P}}_{\dot L}(\varphi)=\exp\left(\,\,\int\limits_{\R^d} \psi(\varphi(x))\lambda^d(dx)\right)
\end{align*}
for every $\varphi \in \mathcal{D}(\R^d)$, where $\psi:\R\to \C$ is a L\'evy exponent, i.e. there exist $a\in\R^+$, $\gamma\in\R$ and $\nu$ a L\'{e}vy-measure, such that   
    \begin{align*}
        \psi(z) = i\gamma z- \frac{1}{2} a z^2 + \int\limits_{\R}(e^{i x z}-1-i x z \mathds{1}_{|x|\le 1})\nu(dx).
    \end{align*}
The function $\psi$ is uniquely characterized by the triplet $(a,\gamma, \nu)$ known as the \emph{characteristic triplet}.
\end{definition}

The existence of the L\'evy white noise was shown in [\ref{Gelfand}]. Another possible way to construct L\'evy white noise would be as an independently scattered random measures, i.e. a random process whose test functions are indicator functions and are independently scattered when two indicator functions with disjoint supports define independent random variables (see B.S. Rajput and J. Rosinski [\ref{Rajput}]). In [\ref{Fageot}] J. Fageot and T. Humeau unified these two approaches by extending the L\'evy white noise, defined as generalized random processes, to independently scattered random measures. This connection led to results in [\ref{Fageot}], which made it possible to extend the domain of definition of L\'evy white noise to some Borel-measurable functions $f : \R^d\to \R$. We say that the function $f$ is in the domain of $\dot{L}$ if there exists a sequence of elementary functions $f_n$ converging almost everywhere to $f$ such that $\langle \dot{L}, f_n \mathds{1}_A\rangle$ converges in probability for $n\to \infty$ for every Borel set $A$ and set $\langle \dot{L},f \rangle $ as the limit in probability of $\langle \dot{L},f_n \rangle $ for $n\to\infty$, where $\langle \dot{L},f_n \rangle $ is defined by $\sum_{j=1}^m a_j \langle \dot{L},\mathds{1}_{A_j} \rangle $ for a elementary function $f_n:=\sum_{j=1}^m a_j \mathds{1}_{A_j}$, see also [\ref{Fageot}, Definition 3.6]. For the maximal domain of the L\'evy white noise $\dot{L}$ we write $D(\dot{L})$. By setting $L(A):=\langle \dot{L}, \mathds{1}_{A} \rangle $ for bounded Borel sets $A$, the extension of a L\'evy white noise $\dot{L}$ can be identified with a L\'evy basis $L$ in the sense of Rajput and Rosinski [\ref{Rajput}], see [\ref{Fageot}, Theorem 3.5 and Theorem 3.7]. As a L\'evy basis can be identified with a L\'evy white noise in a canonical way, i.e. $\langle \dot{L},\varphi\rangle := \int_{\R^d} \varphi(x) dL(x)$ for $\varphi\in \mathcal{D}(\R^d)$, we make no difference between a L\'evy white noise and a L\'evy basis. In particular, a Borel-measurable function $f : \R^d\to \R$ is in $D(\dot{L})$ if and only if $f$ is integrable with respect to the L\'evy basis $L$ in the sense of Rajput and Rosinski [\ref{Rajput}], see [\ref{Fageot}, Definition 3.6].\\

\begin{definition}(see [\ref{Grafakos}, Definition 1.1.1.])
For a measurable function $f\in L^0(\R^d)$ we define the \emph{distribution function} of $f$ as
\begin{align*}
    d_f(\alpha)=\lambda^d(\{ x\in \R^d: |f(x)|>\alpha \}) \textrm{,  }\alpha>0.
\end{align*}
\end{definition}

With the aid of the distribution function we can now obtain a sufficient condition for the existence of the generalized random process $s$ defined by $s(\varphi) =\langle \dot{L}, G(\varphi)  \rangle $, where $G:\R^m\times \R^d\to \R$ is a suitable kernel. This will be crucial in Section 4 for proving the existence of generalized processes as solutions to stochastic partial differential equations as in (\ref{pdov}).

\begin{theorem}\label{theo1}
Let $\dot{L}$ be a L\'evy white noise on $\R^m$ with characteristic triplet $(a,\gamma,\nu)$ and $G:\R^m\times \R^d \to\R$ be a measurable function. Define for every $x\in\R^m$ and $R>0$
\begin{align*}
G_R(x)&:=\int\limits_{B_R(0)}|G(x,y)|\lambda^d(dy)\\
\intertext{ and }
h_R(x) &:=x \int\limits_0^{1/x} d_{G_R}(\alpha)\lambda^1(d\alpha)\textrm{ for }x>0.
\end{align*} 
Assume that $G_R\in L^1(\R^m)\cap L^2(\R^m)$ and
\begin{align}
\label{ass1}\int\limits_{\R} \mathds{1}_{|r|>1} h_R(|r|)\nu(dr)<\infty
\end{align} 
for every $R>0$. Then for $\big(G(\varphi)\big)(x):= \int\limits_{\R^d} G(x,y)\varphi(y)\lambda^d(dy)$ we have that
\begin{align*}
s(\varphi):=\langle\dot{L},G(\varphi)\rangle, \quad \varphi\in \mathcal{D}(\R^d)
\end{align*}
defines a generalized random process.
\end{theorem}

\begin{proof}
The proof is similar to that of [\ref{Berger}, Theorem 3.4], hence we only mention the needed modifications. We need to show that $G( \varphi) \in D(\dot{L})$ and $\langle \dot{L}, G(\varphi_n)\rangle\to \langle \dot{L}, G(\varphi)\rangle$ as $n\to\infty$ in probability for a sequence $(\varphi_n)_{n\in \N}$ converging to $\varphi$ in $\mathcal{D}(\R^d)$. As $\langle \dot{L},G( \cdot) \rangle$ is linear, this is equivalent to check that $\langle \dot{L}, G(\varphi_n-\varphi)\rangle\to 0$ as $n\to\infty$ in probability (see [\ref{Fageot}], Theorem 3.10.). Now given Theorem 2.7 in [\ref{Rajput}], we have to show 
\begin{align}
&\label{david1}\int\limits_{\R^m} \big|\gamma \big(G(\varphi_n)\big)(s) + \int\limits_{\R} r \big(G(\varphi_n)\big)(s)\left(\mathds{1}_{|r(G(\varphi_n))(s)|\le 1}-\mathds{1}_{|r|\le 1}\right)\nu(dr)\big|\lambda^m(ds)\to 0,\\
 &\label{david2}\int\limits_{\R^m}  \int\limits_{\R} \min \big(1, \big|r \big(G(\varphi_n)\big)(s)\big|^2\big)\nu(dr)\lambda^m(ds)\to 0 \quad \textrm{  and }\\
&\label{david3} a^2\int\limits_{\R^m}\big| \big(G(\varphi_n)\big)(s) \big|^2\lambda^m(ds)\to 0
\end{align}
as $n\to\infty$ if $\varphi_n\to 0$ for $n\to \infty$ in $\mathcal{D}(\R^d)$.\\
In the following we give a pointwise upper bound for $G(\varphi)$. Let therefore be $R>0$ such that supp$(\varphi_n)\subset B_r(0)$ for some $r<R$. Then it holds for every $x\in \R^m$
\begin{align}
    \nonumber \big| \big(G(\varphi_n)\big)(x)| &\le \int\limits_{\R^d}  \big|G(x,y)\varphi_n(y)|\lambda^d(dy)\\ 
    &= \int\limits_{B_R(0)} |G(x,y)\varphi_n(y)|\lambda^d(dy)
    \le G_R(x)\|\varphi_n\|_{\infty}\label{l2}.
\end{align}
Now we show (\ref{david1}). Since $G_R\in L^1(\R^m) $, we have 
\begin{align*}
    \int\limits_{\R^m} \big|\gamma \big(G(\varphi_n)\big)(x)\big|\lambda^m(dx)\le |\gamma|\, \|\varphi_n\|_{\infty}\|G_R\|_{L^1(\R^m)}\to 0
\end{align*}
for $n\to \infty$. Furthermore, we obtain with (\ref{l2}) for $\alpha>0$ 
\begin{align}
     \nonumber d_{G(\varphi_n)}(\alpha)=&\lambda^m\left( \{x\in\R^m: |(G(\varphi_n))(x)|> \alpha \}\right)\\
\le& \lambda^m\left(\big\{ x\in\R^m: |G_R(x)|> \frac{\alpha}{\|\varphi_n\|_{\infty}}\big\} \right)=d_{G_R}\left(\frac{\alpha}{\|\varphi_n\|_{\infty}}\right)\label{dist1}.
\end{align}
Since $G_R\in L^2(\R^m)$ we have
\begin{align}
    \nonumber \int\limits_{\R^m} | \big(G(\varphi_n)\big)(x)|\mathds{1}_{| (G(\varphi_n))(x)|> \frac{1}{|r|}}\lambda^m(dx)
    &\le \int\limits_{\R^m} |G\big(\varphi_n\big)(x)|^2 |r| \lambda^m(dx) \\
    &\le  \|\varphi_n\|^2_{\infty} \|G_R\|^2_{L^2(\R^m)}|r|.\label{U1}
\end{align}
Now we get (\ref{david1}) with similar arguments as in the proof of Theorem 3.4 of [\ref{Berger}], where we use (\ref{U1}) instead of the Young Inequality.\\
Since it holds
\begin{align*}
    \|\big(G(\varphi_n)\big)(x)\|^2_{L^2(\R^m)} \le  \|\varphi_n\|^2_{\infty} \|G_R\|^2_{L^2(\R^m)} \to 0
\end{align*}
for $n\to \infty$, we get (\ref{david2}) and (\ref{david3}) again with the same arguments as in the proof of Theorem 3.4 in [\ref{Berger}]. Hence $G(\varphi_n)\to G(\varphi)$ in $D(\dot{L})$ as $n\to\infty$.
\end{proof}

In Theorem \ref{theo1} we assumed that $G_R\in L^1(\R^m)\cap L^2(\R^m)$. In the following Proposition we will show that, if the L\'evy white noise has no Gaussian part and it holds $ \int_{\R} |r|^{\beta} \mathds{1}_{|r|\le 1} \nu(dr)<\infty$, for $\beta\in (1,2)$, then we can assume $G_R\in L^1(\R^m)\cap L^{\beta}(\R^m)$ instead.

\begin{proposition}
Let $G:\R^m\times \R^d\to \R$ be a measurable function and for $R>0$ let $G_R$ and $h_R$ be defined as in Theorem \ref{theo1}. Furthermore, let $\dot{L}$ be a L\'evy white noise on $\R^m$ with characteristic triplet $(0,\gamma,\nu)$ such that (\ref{ass1}) holds. If further $G_R \in L^1(\R^m)\cap L^{\beta}(\R^m)$ for some $\beta\in (1,2)$ and
\begin{align*}
    \int\limits_{\R} |r|^{\beta} \mathds{1}_{|r|\le 1} \nu(dr)<\infty,
\end{align*}
then
\begin{align*}
s(\varphi):=\langle\dot{L},G(\varphi)\rangle, \quad \varphi\in \mathcal{D}(\R^d)
\end{align*}
defines a generalized random process, where $G(\varphi)$ is defined as in Theorem \ref{theo1}.
\end{proposition}

\begin{proof}
Again, the proof is similar to that of [\ref{Berger}, Theorem 3.4] and hence we only mention the needed modifications. As $G_R\in L^1(\R^m)$ we only have to consider the terms which were estimated with $\|G_R\|_{ L^2(\R^m)}$ as can be seen from the proof of [\ref{Berger}, Theorem 3.4]. These are 
\begin{align}
     &\int\limits_{\R} |r| \mathds{1}_{|r|\le 1} \int\limits_{\R^m} | \big(G(\varphi_n)\big)(x)|\mathds{1}_{| (G(\varphi_n))(x)|> \frac{1}{|r|}}\lambda^m(dx) \nu(dr)\label{annahmen1}
\intertext{ and }
     & \int\limits_{\R} \int\limits_{\R^m}  |r|^2 |\big(G(\varphi_n)\big)(x)|^2\mathds{1}_{| r(G(\varphi_n))(x)|\le  1} \mathds{1}_{|r|\le 1} \lambda^m(dx) \nu(dr)\label{annahmen2}.
\end{align}
and we have to show that they converge to $0$ as $\varphi_n\to 0$ in $\mathcal{D}(\R^d)$. We have
\begin{align*}
    \int\limits_{\R^m} | \big(G(\varphi_n)\big)(x)|\mathds{1}_{| (G(\varphi_n))(x)|> \frac{1}{|r|}}\lambda^m(ds)
    &\le \|(G(\varphi_n))\|^{\beta}_{L^{\beta}(\R^m)} |r|^{\beta-1} \le  \|\varphi_n\|^{\beta}_{\infty} \|G_R\|^{\beta}_{L^{\beta}(\R^m)}|r|^{\beta-1}.
\end{align*}
So it follows that the term (\ref{annahmen1}) converges to $0$ as $\varphi_n\to 0$ in $\mathcal{D}(\R^d)$. Furthermore, it holds
\begin{align*}
    &\int\limits_{\R}\int\limits_{\R^m}  |r|^2 |\big(G(\varphi_n)\big)(x)|^2\mathds{1}_{| r(G(\varphi_n))(x)|\le  1} \mathds{1}_{|r|\le 1} \lambda^m(dx) \nu(dr)\\
  = & \int\limits_{\R} \int\limits_{\R^m} |r|^{\beta} |\big(G(\varphi_n)\big)(x)|^{\beta}  |r|^{2-\beta}|\big(G(\varphi_n)\big)(x)|^{2-\beta}\mathds{1}_{| r(G(\varphi_n))(x)|\le  1} \mathds{1}_{|r|\le 1} \lambda^m(dx) \nu(dr)\\
    \le & \int\limits_{\R} |r|^{\beta}\mathds{1}_{|r|\le 1} \nu(dr)  \|\varphi_n\|^{\beta}_{\infty} \|G_R\|^{\beta}_{L^{\beta}(\R^m)}.
\end{align*}
This shows that the term (\ref{annahmen2}) converges to $0$ as $\varphi_n\to 0 $ in $\mathcal{D}(\R^d)$ and the rest of the proof follows with similar arguments as mentioned in the proof of Theorem \ref{theo1}.
\end{proof}

When $G_R\notin L^1(\R^m)$ we can still obtain a generalized process $s$ under some extra conditions. Similar to Theorem 3.5 in [\ref{Berger}] we have

\begin{theorem}\label{theo2}
Let $G:\R^m\times\R^d\to\R$ be a measurable function such that $G_R\in L^2(\R^m)$, where $G_R$ and $G(\varphi)$, $\varphi\in\mathcal{D}(\R^d)$ are defined as in Theorem \ref{theo1}. If the first moment of the L\'evy white noise $\dot{L}$ on $\R^m$ with characteristic triplet $(a,\gamma, \nu) $ vanishes, i.e. $\mathbb{E}|\langle \dot{L},\varphi \rangle|< \infty$ and $\mathbb{E}\langle \dot{L},\varphi\rangle = 0$ for every $\varphi\in \mathcal{D}(\R^d)$,  then $s:\mathcal{D}(\R^d)\to L^0(\Omega)$ defined by 
\begin{align*}
    s(\varphi): = \langle \dot{L}, G(\varphi) \rangle 
\end{align*}
is a generalized random process if
\begin{align}
    \int\limits_{\R} \mathds{1}_{|r|>1} |r| \int\limits_{\frac{1}{|r|}}^{\infty} d_{G_R}(\alpha)\lambda^1(d\alpha)\nu(dr)<\infty \label{g1}
    \intertext{and}
     \int\limits_{\R} \mathds{1}_{|r|>1} |r|^2 \int\limits_{0}^{\frac{1}{|r|}} \alpha d_{G_R}(\alpha)\lambda^1(d\alpha)\nu(dr)<\infty\label{g2}
\end{align}
for all $R>0$.
\end{theorem}

\begin{proof}
Let $(\varphi_n)_{n\in\N}$ be a sequence converging to $0$ in $\mathcal{D}(\R^d)$ such that supp $\varphi_n\subset B_R(0)$ for some $R>0$ and all $n\in\N$. This proof follows with the same arguments as in the proof of [\ref{Berger}, Theorem 3.5], where we use $G(\varphi_n)$ instead of $G\ast \varphi_n$ and $\|\varphi_n\|_{\infty} \|G_R\|_{L^2(\R^m)}<\infty$ instead of $\|G\ast\varphi_n\|_{L^2(\R^d)}<\infty$.
\end{proof}

\begin{example}\label{example1}
Let $d\ge 1$, $q\in [1,2)$ and $\frac{d}{2}<p<\frac{d}{q}$. We consider $G:\R^d\times \R^d\to\R$ such that it holds 
\begin{align*}
	|G(x,y)|\|x-y\|^p \le w(y)
\end{align*}
for all $x,y\in \R^d$, where $w\in  L^{q*}_{loc}(\bR^d)$ with $q^*=\frac{q}{q-1}$.
With the H\"older's inequality we conclude for $R>0$ and $x\in \bR^d$
\begin{align}
    \nonumber G_R(x):&=\int\limits_{B_R(0)} |G(x,y)| \lambda^d(dy) \\
    \nonumber&\le \left( \,\,\int\limits_{B_R(0)} \|x-y \|^{-qp} \lambda^d(dy) \right)^{1/q} \left(\,\, \int\limits_{B_R(0)} |w(y)|^{q^*} \lambda^d(dy)\right)^{1/q*}\\
    \label{1002}&\le C(w,q,p,d,R) \min\{1,\|x\|^{-p}\}.
\end{align}
We obtain that
\begin{align*}
    \|G_R\|_{L^2(\R^d)}<\infty.
\end{align*}
Furthermore, we observe for a L\'evy white noise $\dot{L}$ with characteristic triplet $(a,\gamma,\nu)$ that
\begin{align*}
    &\int \limits_0^{\frac{1}{|r|}} \alpha d_{G_R}(\alpha) \lambda^1(d\alpha) \le C  \int\limits_0^{\frac{1}{|r|}} \alpha(1+\alpha^{-\frac{d}{p}})  \lambda^1(d\alpha)
    = \tilde{C} (|r|^{-2}+|r|^{\frac{d}{p}-2})
\intertext{and}
    &\int\limits_{\R} \mathds{1}_{|r|>1} |r|^2\int \limits_0^{\frac{1}{|r|}} d_{G_R}(\alpha) \lambda^1(d\alpha) \nu(dx) \le \int\limits_{\R} \mathds{1}_{|r|>1} \tilde{C}(1+|r|^{\frac{d}{p}}) \nu(dx),
\end{align*}
where $\tilde{C}>0$. If the L\'evy white noise $\dot{L}$ has vanishing first moment then it follows from [\ref{Sato}, Example 25.12] that (\ref{g1}) is satisfied. So if additionally $\dot{L}$ satisfies
\begin{align*}
	\int\limits_{\R} \mathds{1}_{|r|>1}|r|^{\frac{d}{p}} \nu(dx)< \infty 
\end{align*}
then it follows from Theorem \ref{theo2} that
\begin{align*}
    s:\mathcal{D}(\bR^d)\to L^0(\Omega),\,\varphi\mapsto s(\varphi):=\langle \dot{L}, G(\varphi)\rangle
\end{align*}
defines a well-defined generalized random process.
\end{example}

\subsection{Moment properties}\label{section 4.4}
\noindent Next we show, that if the L\'evy white noise $\dot{L}$ has finite $\beta>0$ moment, then so has the generalized random process $s(\varphi) = \langle \dot{L}, G(\varphi) \rangle$, $\varphi\in \mathcal{D}(\R^d)$.
\begin{theorem}\label{momentprop}
Let $G:\bR^m\times\bR^d\to \bR$ be a measurable function different from $0$ and $\dot{L}$ be a L\'{e}vy white noise on $\R^m$ with characteristic triplet $(a,\gamma,\nu)$ and assume that $\langle s,\varphi\rangle:=\langle\dot{L},G(\varphi)\rangle$, $\varphi\in \mathcal{D}(\R^d)$ is a well-defined generalized random process. Let $\beta>0$
\begin{itemize}
    \item[i)] If $0<\beta< 2$ assume that $G_R\in L^\beta (\bR^m)\cap  L^2 (\bR^m)$ with $G_R$ as defined in Theorem \ref{theo1}. If $\dot{L}$ has finite $\beta$-moment, then so has $s$. If $\beta\ge2$ it is sufficient to assume that $G_R\in L^\beta (\bR^d)$.
    \item[ii)] If $s$ has finite $\beta$-moment, then $\dot{L}$ has also finite $\beta$-moment.
\end{itemize}
\end{theorem}

\begin{proof}
From [\ref{Rajput}, Theorem 2.7] we know that the L\'evy measure of the random variable $\langle s,\varphi\rangle$ is given by
\begin{align*}
    \nu_{s(\varphi)}(B) = \int\limits_{\R^m}\int\limits_{\R} \mathds{1}_{B\setminus \{0\} } \left(r G(\varphi)(x)\right) \nu(dr)\lambda^m(dx).
\end{align*}
Then $\langle s,\varphi \rangle$ has finite $\beta$-moment if and only if $\int\limits_{|z|>1} |z|^{\beta} \nu_{s(\varphi)}(dz)<\infty$.\\ \noindent i) Let $\dot{L}$ have finite $\beta$-moment and assume at first that $0< \beta< 2$. We calculate with (\ref{l2}) that
\begin{align*}
    \int\limits_{|z|>1} |z|^{\beta}\nu_{s(\varphi)}(dz)
    =&\int\limits_{\R} |r|^{\beta}\int\limits_{|G(\varphi)(x)|>\frac{1}{|r|}} |G(\varphi)(x)|^{\beta}\lambda^m(dx)\nu(dr)\\
    \le & 
    \int\limits_{|r|\le 1} |r|^{2}\int\limits_{|G(\varphi)(x)|>\frac{1}{|r|}} |G(\varphi)(x)|^{2}\lambda^m(dx)\nu(dr)\\
    &+
    \int\limits_{|r|>1} |r|^{\beta}\int\limits_{\R^m} |G_R(x)|^{\beta} \|\varphi\|^{\beta}_{\infty} \lambda^m(dx)\nu(dr)\\
    \le & \int\limits_{|r|\le 1} |r|
    ^2\nu(dr)\|G_R\|_{L^2(\bR^m)}^2\|\varphi\|_{\infty}^{2}+\|G_R\|_{L^{\beta}(\R^m)}^{\beta}  \int\limits_{|r|>1} |r|^{\beta} \nu(dr)\|\varphi\|_{\infty}^{\beta}<\infty,
\end{align*}
where $R>0$ is such that $\supp \varphi\subset B_R(0)$.\\
If $\beta\ge 2$ we obtain by similar arguments as above that
\begin{align*}
    \int\limits_{|z|>1} |z|^{\beta}\nu_{s(\varphi)}(dz)\le \|\varphi\|_{\infty}^{\beta}\|G_R\|_{L^{\beta}(\R^m)}^{\beta}  \int\limits_{\bR} |r|^{\beta} \nu(dr),
\end{align*}
which is indeed finite.\\
\noindent ii) Assume that $s$ has finite $\beta$-moment and that $G$ is different from $0$. So we know that there exists a function $\varphi \in \mathcal{D}(\bR^d)$  such that
\begin{align*}
\int_{\bR^d} |G(\varphi)(x)|\lambda^m(dx)>0,
\end{align*}
hence there exists an $r_0>1$ with
\begin{align*}
 \int_{|G(\varphi)|>1/r_0} |G(\varphi)(x)|^\beta\lambda^m(dx)>0.   
\end{align*}
We conclude
\begin{align*}
    \infty> \int\limits_{|z|>1} |z|^{\beta}\nu_{s(\varphi)}(dz)
    =&\int\limits_{\R} |r|^{\beta}\int\limits_{|G(\varphi)(x)|>\frac{1}{|r|}} |G(\varphi)(x)|^{\beta}\lambda^m(dx)\nu(dr)\\
    \ge& \int\limits_{|r|>r_0} |r|^{\beta}\int\limits_{|G(\varphi)(x)|>\frac{1}{|r|}} |G(\varphi)(x)|^{\beta}\lambda^m(dx)\nu(dr)\\
    \ge& \int\limits_{|r|>r_0} |r|^{\beta}\nu(dr)\int\limits_{|G(\varphi)(x)|>\frac{1}{|r_0|}} |G(\varphi)(x)|^{\beta}\lambda^m(dx),
\end{align*}
hence $\int\limits_{|r|>r_0} |r|^{\beta} \nu(dr)<\infty$ so that $\dot{L}$ has finite $\beta$-moment.
\end{proof}

\section{Second order elliptic partial differential equations  driven by L\'{e}vy white noise}

\subsection{Second order elliptic partial differential equations in divergence form driven by Levy white noise} 

In this section we discuss elliptic partial differential operators of second order with variable coefficients in divergence form, i.e. partial differential operators $p(x,D)$ of the form
\begin{align}\label{eqeli}
    p(x,D)u=-\sum\limits_{i,j=1}^d \partial_i (a_{ij}(x) \partial_j u) 
    =-\textrm{div} ((A(x)\nabla u),
\end{align}
where $A(x)=(a_{ij}(x))_{i,j=1}^d\in C^\infty(\bR^d,\bR^{d\times d})$ is a \emph{uniformly elliptic matrix}, i.e. there exists a $C>0$ such that
\begin{align*}
    C^{-1} \|\xi\|^2\le \xi^T A(x)\xi\le C\|\xi\|^2\textrm{ for all }\xi \in\bR^d.
\end{align*}
Now let $\dot{L}$ be a L\'{e}vy white noise on $\R^d$ with characteristic triplet $(a,\gamma,\nu)$ and $p(x,D)$ be a partial differential operator (PDO) of the form \eqref{eqeli}. 
We say that a generalized stochastic process $s:\mathcal{D}(\bR^d)\to L^0(\Omega)$ is a \emph{generalized solution} of the equation
\begin{align*}
    p(x,D) s=\dot{L},
\end{align*}
if it holds
\begin{align*}
    \langle s,p(x,D)^*\varphi\rangle =\langle \dot{L},\varphi\rangle \quad\textrm{for all }\varphi\in\mathcal{D}(\bR^d),
\end{align*}
where $p(x,D)^*$ is the adjoint of $p(x,D)$, i.e.
\begin{align*}
    p(x,D)^*u=-\sum\limits_{i,j=1}^d \partial_i (a_{ji}(x) \partial_j u).
\end{align*}
In the first theorem we derive sufficient conditions for the existence of such a solution in terms of the characteristic triplet $(a,\gamma,\nu)$, which is just a a simple extension of the Laplacian case. Afterwards we discuss stationarity of these generalized processes, e.g. if the coefficients are $y-$periodic for some $y\in\bR^d$, then $s$ is $y-$periodically stationary. We assume for the complete section that the coefficients of $p(x,D)$  are in $C^\infty(\bR^d)$. 
\begin{theorem}\label{ellivar}
Let $\dot{L}$ be a L\'{e}vy white noise on $\R^d$ with characteristic triplet $(a,\gamma,\nu)$ with vanishing first moment and $p(x,D)$ be a 
PDO of the form \eqref{eqeli}.
The stochastic partial differential equation
\begin{align}\label{ellicarma}
    p(x,D)s=\dot{L}
\end{align}
has a generalized solution $s:\mathcal{D}(\bR^d)\to L^0(\Omega)$, if $d\ge 5$ and
\begin{align*}
    \int_{|r|>1} |r|^{d/(d-2)}\nu(dr)<\infty.
\end{align*}
\end{theorem}

\begin{proof}
By [\ref{Littman}, Chapter 10] there exists a locally integrable left inverse $E:\bR^d\times\bR^d\to \bR$ of the operator $p(x,D)^*$ such that for all $\varphi \in\mathcal{D}(\bR^d)$
\begin{align*}
    E(p(\cdot,D)^*\varphi)(x):=\int_{\bR^d}E(x,y)p(y,D)^*\varphi(y) dy=\varphi(x)\textrm{ for all }x\in\bR^d.
\end{align*}
Moreover, there exists an $N\in\bN$ such that
\begin{align*}
    N^{-1}\|x-y\|^{2-d}\le E(x,y)\le N\|x-y\|^{2-d}\textrm{ for all }x\neq y.
\end{align*}
We set 
\begin{align*}
\langle s,\varphi\rangle:= \langle \dot{L}, E (\varphi)\rangle
\end{align*}
and from Example $\ref{example1}$ with $w=1$, $p=d-2$ and $q=1$ (observe that $d\ge 5$) it follows that
\begin{align*}
s:\mathcal{D}(\bR^d)&\to L^0(\Omega),\\
\langle s,\varphi\rangle&:=\langle \dot{L},E(\varphi)\rangle,\quad\varphi\in\mathcal{D}(\bR^d),
\end{align*}
defines a generalized process. Moreover, $s$ is a solution of the equation \eqref{ellicarma}, as
\begin{align*}
\langle s,p(x,D)^*\varphi\rangle =\langle \dot{L},E(p(x,D)^*\varphi) \rangle=\langle \dot{L},\varphi \rangle
\end{align*}
for every $\varphi\in \mathcal{D}(\bR^d)$.
\end{proof}
The solution $s:\mathcal{D}(\bR^d)\to L^0(\Omega)$ is not unique, which is quite clear. For example, let $p(x,D)=-\Delta$ and define
\begin{align*}
    \langle s',\varphi\rangle:= \langle s,\varphi\rangle + \int\limits_{\bR^d} (x_1^2-x_2^2)\varphi(x) \lambda^d(dx),
\end{align*}
where $s$ is the solution constructed in Theorem \ref{ellivar} for the equation
\begin{align}\label{example}
    -\Delta s=\dot{L}.
\end{align}
Then it is easy to see that $s'$ is also a solution of \eqref{example}.

\begin{remark}\label{remark5.5}
We assumed that the coefficients of the partial differential operator $p(x,D)$ are infinitely often differentiable, but this is not necessary. It would be sufficient if $a_{ij}\in C^1(\bR^d)$ for all $i,j\in \{1,\dotso,d\}$.
\end{remark}

\begin{remark}\label{genrem}
The method above can also be used to find solutions of SDPEs of the form
\begin{align*}
-\textrm{div}(A\nabla u)+b\cdot \nabla u+Vu=\dot{L}
\end{align*}
under some suitable assumptions for the functions $A,\,b$ and $V$, as the fundamental solution $E$ of the elliptic operator above can be bounded from above by a constant times $\|x-y\|^{d-2}$ for all $x\neq y$. For a very general result see [\ref{Davey}]. Observe that in the most general case the fundamental solution solves the equation only in the weak sense. We will discuss in the next section what we understand under a weak solution.
\end{remark}

As a next step we discuss stationarity properties, which depend heavily on the matrix $(a_{ij}(x))_{i,j=1}^d$. For example, if $a_{ij}:\bR^d\to \bR$ is constant, it is easily seen that $E(x,y)=E(x-y)$ for all $x\neq y$ and hence we observe that the constructed solution $s:\mathcal{D}(\bR^d)\to L^0(\Omega)$ in Theorem \ref{ellivar} is stationary. 

\begin{definition}
A generalized process $s$ on $\mathcal{D}(\bR^d)$ is called \emph{periodic with period $l\in \R^d$}, if $s(\cdot+l)$ has the same law as $s$, and \emph{stationary} if $s$ is periodic for every period $l\in \R^d$. Here, $s(\cdot+l)$ is defined by
\begin{align*}
  \langle s(\cdot+l),\varphi\rangle:=\langle s, \varphi(\cdot-l)\rangle \textrm{ for every }\varphi\in\mathcal{D}(\R^d). 
\end{align*}
\end{definition}

\begin{remark}
Let $G:\R^m\times \R^d\to \R$ be a measurable function which fulfills the assumptions of Theorem \ref{theo1} with $m=d$. Assume that $G(x,y+l)= G(x+l,y)$ for all $x,y\in\R^d$ and for some $l\in\R^d$. Then it is easily seen that for $\varphi\in \mathcal{D}(\R^d)$
\begin{align*}
    \left( G \varphi(\cdot-l) \right)(x) = \int\limits_{\R^d} G(x,y) \varphi(y-l) \lambda^d(dy) = \left( G\varphi \right)(x+l),
\end{align*}
hence the generalized process $s$ defined in Theorem \ref{theo1} satisfies
\begin{align*}
    \langle s(\cdot+l), \varphi \rangle = \langle s,\varphi(\cdot-l) \rangle = \langle \dot{L}, G\varphi(\cdot-l) \rangle = \langle \dot{L}, \left(G\varphi  \right)(\cdot+l) \rangle = \langle \dot{L}(\cdot-l),G\varphi \rangle.
\end{align*}
Since $\dot{L} \overset{d}{=} \dot{L}(\cdot-l)$ it follows that in this case the process $s$ is periodic with period $l$. Observe that $(s(\varphi(\cdot+ly)))_{y\in \mathbb{Z}}$ is then a stationary process for all $\varphi\in \mathcal{D}(\R^d)$. Therefore, these models seem to be useful in statistics to model periodic processes or random fields. In the case that $G(x,y+l) = G(x+l,y)$ for all $l,x,y\in\R^d$, we see that $s$ will be stationary. 
\end{remark}

\begin{proposition}\label{propstat}
Let $p(x,D):\mathcal{D}(\bR^d)\to C(\bR^d)$ be an elliptic partial differential operator of the form \eqref{eqeli}, $d\ge 5$ and assume that the matrix-valued function $A:\bR^d\to \bR^{d\times d}$ is periodic with period $y\in\bR^d$, i.e. $A(x+y)=A(x)$ for all $x\in\bR^d$. Let $\dot{L}$ be a L\'{e}vy white noise such that it satisfies the assumption of Theorem \ref{ellivar}. Then there exists a solution $s:\mathcal{D}(\bR^d)\to L^0(\Omega)$ of $p(x,D)s=\dot{L}$, which is periodically stationary with period $y$.
\end{proposition}

\begin{proof}
It is enough to show that
\begin{align*}
    E(\varphi(\cdot+y))(x)=E(\varphi)(x+y),
\end{align*}
where $E$ is again the fundamental solution of the operator $p(x,D)^*$. The assertion follows then from the stationarity of the L\'{e}vy white noise $\dot{L}$. We see that
\begin{align*}
    p(x,D)^*E(\varphi(\cdot+y))(x)=\varphi(x+y)
\end{align*}
and
\begin{align*}
    p(x,D)^*E(\varphi)(x+y)=p(x+y,D)^*E(\varphi)(x+y)=\varphi(x+y),
\end{align*}
so $ E(\varphi)(\cdot+y):\bR^d\to \bR$ and $E(\varphi(\cdot+y)):\bR^d\to \bR$ solve the same elliptic equation. By construction it holds
\begin{align}\label{1001}
   \lim\limits_{|x|\to\infty}
   |(E(\varphi)(x+y)-E(\varphi(\cdot+y))(x))|&=0\textrm{ and }\\
       \nonumber p(x,D)^*(E(\varphi)(x+y)-E(\varphi(\cdot+y))(x))&=0\textrm{ for all }x\in\bR^d.
\end{align}
where \eqref{1001} follows from \eqref{l2} and \eqref{1002}.
By the maximum principle for uniformly elliptic equations we obtain $E(\varphi)(x+y)-E(\varphi(\cdot+y))(x)=0$ for all $x\in\bR^d$, hence we obtain that $s:\mathcal{D}(\bR^d)\to L^0(\Omega)$ is periodically stationary.
\end{proof}
From this result we can construct a stationary process on a certain group as long as the coefficients of the partial differential operator satisfy some periodicity condition.
\begin{corollary}
    Let $(\mathcal{G},+)$ be a subgroup of $(\bR^d,+)$ and $p(x,D):\mathcal{D}(\bR^d)\to C(\bR^d)$ be an elliptic partial differential operator of the form \eqref{eqeli} and assume that the matrix-valued function $A:\bR^d\to \bR^{d\times d}$ is periodic with period $y\in\mathcal{G}$ for all $y\in \mathcal{G}$. Let $\dot{L}$ be a L\'{e}vy white noise satisfying the assumption of Theorem \ref{ellivar} and $s$ be the generalized solution of $p(x,D)s=\dot{L}$ constructed in Theorem \ref{ellivar}. Then for every $\varphi\in\mathcal{D}(\bR^d)$ the process
\begin{align*}
        (s_{\varphi}(y))_{y\in\mathcal{G}}:=(\langle s,\varphi(\cdot+y)\rangle)_{y\in \mathcal{G}}
\end{align*}
is a stationary process in $\mathcal{G}$.
\end{corollary}

\begin{proof}
This is a direct consequence of Proposition \ref{propstat}.
\end{proof}

\subsection{The generalized and mild solutions of the electric Schr\"odinger equation driven by L\'evy white noise}
We saw in Remark \ref{genrem} before, that we can find generalized solutions of stochastic partial differential equations given by
\begin{align}
    -\textrm{div}(A(x)\nabla u)+V(x)u=\dot{L}\label{schroderEq}
\end{align}
for suitable $A$ and $V$ by assuming that the dimension $d\ge 5$, the first moment of the L\'{e}vy white noise vanishes and under the moment condition
\begin{align*}
\int\limits_{|r|>1} |r|^{d/(d-2)}\nu(dr)<\infty.
\end{align*}
In the case that $V$ lies in a Reverse H\"older class these assumptions seem to be not necessary. We show that we find generalized and mild solutions in dimension $3$ under much weaker conditions. At first we introduce the Reverse H\"older class $RH_p(\bR^d)$ and if $V$ is in this class, the moment assumption reduces to some kind of a logarithm moment condition (dependent on $V$), which is very similar to the case that $V$ is a positive constant. We first define what is meant by a mild solution of (\ref{schroderEq}).
\\
We call $E:\R^d\times\R^d\to\R$ a weak fundamental solution of the generalized electric Schr\"odinger operator
\begin{align*}
    -\textrm{div}(A(x)\nabla u) + V(x) u,
\end{align*}
if $E(\varphi):= \int\limits_{\R^d} E(x,y)\varphi(y)\lambda^d(dy)$ solves 
\begin{align*}
    -  \textrm{div}(A \nabla E(\varphi)) +VE(\varphi) = \varphi
\end{align*}
in the weak sense for all $\varphi\in\mathcal{D}(\R^d)$. 
We set $u(x):=\langle \dot{L}, E(x,\cdot)\rangle$ to be the \emph{mild solution} of \eqref{schroderEq}, if $u(x)$ exists for all $x\in\bR^d$, i.e. if $E(x,\cdot)\in  D(\dot{L})$ for all $x\in\R^d$. Then Theorem \ref{SchroderSol} i) will give a sufficient condition for that to hold.

In the following we define the maximum function $m$ and Agmon distance $\gamma$ of the potential $V$, to apply the estimates of the fundamental solution of the generalized electric Schr\"odinger operator shown in [\ref{Shen}] and [\ref{Mayboroda}].

\begin{definition}
Let $p\ge 1$. A function $\omega\in L^p_{loc}(\R^d)$ with $\omega>0$ a.e. belongs to the \emph{Reverse H\"older class} $RH_p(\R^d)$ if there exists a constant $C$ so that for any ball $B\subset \R^d$,
\begin{align*}
     \left( \frac{1}{\lambda^d(B)}\int\limits_B \omega(x)^p\lambda^d(dx)\right)^{1/p} \le \frac{C}{\lambda^d(B)} \int\limits_B \omega(x) \lambda^d(dx). 
\end{align*}
Furthermore, we define for $\omega\in RH_p(\R^d)$ the \emph{maximum function} $m(x,\omega)$ by
\begin{align*}
    \frac{1}{m(x,\omega)}:=\sup\left\{r>0:\frac{1}{r^{d-2}}\int\limits_{B(x,r)} \omega(y)dy\le 1\right\}\in (0,\infty)
\end{align*}
and the \emph{distance function}
\begin{align*}
    \gamma(x,y,\omega):= \inf_{\Gamma} \int_0^1 m(\Gamma(t),\omega)|\dot\Gamma(t)| \lambda^1(dt),
\end{align*}
where $\Gamma:[0,1]\to \bR^d$ is absolutely continuous and $\Gamma(0)=x$ and $\Gamma(1)=y$. Moreover, we define for $R>0$ the ball
\begin{align*}
    B^{\omega}(x,R):=\big\{y\in\R^d: \gamma(x,y,\omega)<R \big\}.
\end{align*}
\end{definition}
The set $RH_p(\R^d)$ is closely connected to the space of Muckenhoupt weights $A_s$, $s\ge 1$, where $\omega$ measurable and non-negative is in $A_s$ if
\begin{align*}
\sup_{B \textrm{ ball in }\R^d}\left(\frac{1}{\lambda^d(B)} \int\limits_B \omega(x) \lambda^d(dx)\right)\left(\frac{1}{\lambda^d(B)} \int\limits_B \omega(x)^{-s'/s} \lambda^d(dx)\right)^{s/s'}<\infty,
\end{align*}
where $s'\in \R$ such that $\frac{1}{s}+\frac{1}{s'}=1$. For further information see for example [\ref{Stein}]. Especially it holds that $\omega\in A_s$ for some $s\ge 1$ if and only if there exists a $p>1$ such that $\omega\in RH_p(\R^d)$. 
We see that the set of all positive and measurable functions bounded from above and strictly away from zero  given by
\begin{align*}
   \bigg\{ &f:\R^d\to(0,\infty):\exists C_1,\,C_2>0 \textrm{ such that } C_1 \le f(y)\le C_2  \textrm{ for all }y\in\bR^d \bigg\} 
\end{align*}
is a subset of $RH_p(\R^d)$ for all $p\ge 1$. 
We state now an existence theorem for a mild and generalized solution of the equation
\begin{align*}
    (-\textrm{div}(A\nabla)+V)s=\dot{L},
\end{align*}
where $V$ lies in $RH_{\frac{d}{2}}(\R^d)$ and show that under much weaker moment conditions there exists a generalized solution. We use that the weak fundamental solution $E$ of the operator $p(x,D)$ can be bounded as follows
\begin{align}\label{mayborodaeq}
|E(x,y)|\le C\frac{e^{-k\gamma(x,y,V)}}{\|x-y\|^{d-2}} \quad\textrm{ for all } x,y\in\R^d, x\neq y,
\end{align}
where $k,C>0$, see [\ref{Mayboroda}, Corollary 6.16, page 40]. From now on the constant $k>0$ is fixed and such that (\ref{mayborodaeq}) is satisfied.

\begin{theorem}\label{SchroderSol}
Let $A(x)=(a_{i,j}(x))_{i,j=1}^{d}$ be a real, uniformly bounded and elliptic matrix and $V\in RH_{\frac{d}{2}}(\bR^d)$. Let $\dot{L}$ be a a L\'evy white noise on $\R^d$ with characteristic triplet $(a,\gamma,\nu)$ such that it holds 
\begin{align*}
    \int\limits_{|r|>1} |r| \int\limits_{0}^{1/|r|} \lambda^d\left(B^{V}(0,- \frac{1}{k} \log(\alpha))\right) \lambda^1(d\alpha) \nu(dr)<\infty.
\end{align*}
i) If $d=3$ then there exists a mild solution of
\begin{align*}
     -\emph{div}(A\nabla u )+Vu = \dot{L},
\end{align*}
which is stochastically continuous.\\
ii) If $d\ge 3$ then there exists a generalized solution $s: \mathcal{D}(\R^d) \to L^0(\Omega) $ of 
\begin{align*}
	(-\emph{div}(A \nabla ) + V)s = \dot{L}. 
\end{align*}
iii) Under the assumption that the first moment of the L\'{e}vy white noise exists, the mild solution $u$ from i) gives rise to a generalized solution $s$ of the stochastic partial differential equation $(-\emph{div}(A\nabla) +V)s= \dot{L}$ via
\begin{align*}
    \langle s,\varphi \rangle := \int\limits_{\R^d} u(x)\varphi(x)\lambda^d(dx).
\end{align*}

\end{theorem}
We will prove Theorem \ref{SchroderSol} in Section \ref{secproof}. Here we will calculate the moment condition for $\dot{L}$ for functions which are greater than a positive constant.
\begin{example}
Let $d\ge3$ and $V\in RH_{\frac{d}{2}}(\R^d)$ such that $V>\varepsilon$, where $\varepsilon>0$. We observe that
\begin{align*}
    \int\limits_0^1 m(\Gamma(t),V ) |\dot{\Gamma}(t)| \lambda^1(dt) \ge C\sqrt{\varepsilon} \|y-x\|
\end{align*}
for every path $\Gamma:[0,1]\to\bR^d$ with $\Gamma(0)=x$ and $\Gamma(1)=y$ from which it follows for $0<\alpha\le 1$ that for fixed $k>0$
\begin{align*}
    \lambda^d\left( B^V(0,-\frac{1}{k} \log(\alpha)) \right) \le C_1  \left(\log\left(\frac{C}{\alpha}\right)\right)^d,
\end{align*}
where $C,C_1>0$. Since
\begin{align*}
    \int\limits_0^{1/r} \left(\log\left( \frac{1}{\alpha}\right) \right)^d \lambda^1(d\alpha) = \int\limits_{\log(r)}^{\infty} \beta^d e^{-\beta} \lambda^1(d\beta) =\Gamma(d+1,\log(r))= \frac{d!}{r} \sum\limits_{j=0}^d \frac{(\log( r))^j}{j!},
\end{align*}
where $\Gamma(d+1,\log(r))$ denotes the upper incomplete gamma function, this leads to
\begin{align*}
    &\int\limits_{|r|>1} |r| \int\limits_{0}^{1/|r|} \lambda^d\left(B^{V}(0,- \frac{1}{k} \log(\alpha))\right) \lambda^1(d\alpha) \nu(dr)\\
    \le&\int\limits_{|r|>1}  C_2\log(|r|)^d \nu(dr)+C_3\nu(\bR\setminus[-1,1]),
\end{align*}
where $C_2,C_3>0$.
So if we assume that the L\'evy white noise $\dot{L}$ with characteristic triplet $(a,\gamma,\nu)$ satisfies
\begin{align*}
    \int\limits_{|r|>1} \log(|r|)^d \nu(dr)<\infty
\end{align*}
then the assumptions of Theorem \ref{SchroderSol} are satisfied and we obtain generalized and mild solutions, if $d\ge 3$ or $d= 3$ respectively.
\end{example}
\subsection{Existence and continuity of mild solutions}
In the following we give sufficient conditions for the existence and continuity of a random field $u(x):=(\langle\dot{L},E(x,\cdot)\rangle)_{x\in \bR^m}$, where $E:\R^m\times \R^d\to\R$ is a kernel. This will be used in the proof of Theorem \ref{SchroderSol}, where $E$ is the weak fundamental solution of the generalized electric Schr\"odinger operator. 

\begin{proposition}\label{mild}
Let $\dot{L}$ be a L\'evy basis on $\R^d$ with characteristic triplet $(a,\gamma,\nu)$ and let $E:\R^m\times \R^d\to\R$ be a measurable function. We define for every $x\in\R^m$ a function $h_x:\bR^+\to \bR^+$ by
\begin{align*}
    h_x(r) &:=r \int\limits_0^{1/r} d_{E(x,\cdot)}(\alpha)\lambda^1(d\alpha)\textrm{ for }r>0.
\end{align*}
\begin{itemize}
    \item[i)] Assume that $E(x,\cdot)\in L^1(\R^d)\cap L^2(\R^d)$ for every $x\in\bR^m$ and
\begin{align*}
    \int\limits_{\R} \mathds{1}_{|r|>1} h_x(|r|)\nu(dr)<\infty
\end{align*}
for every $x\in\bR^m$. Then $E(x,\cdot)\in D(\dot{L})$ for every $x\in\bR^m$ and hence the random field $u=(u(x))_{x\in\bR^m}$ given by $u(x):= \langle \dot{L}, E(x,\cdot) \rangle$ for all $x\in \R^m$ exists.
    \item[ii)] Furthermore, if the function $T_E:\bR^m\to L^1(\bR^d)\cap L^2(\bR^d)$ given by $T_E(x):=E(x,\cdot)$ is continuous in $L^1(\R^d)$ and $L^2(\R^d)$ and for every $x\in \bR^m$ there exists an $\varepsilon>0$ such that
    \begin{align*}
    \sup_{x^*\in B_{{\varepsilon}}(x)}\int\limits_{\R} \mathds{1}_{|r|>1} h_{x^*}(|r|)\nu(dr)<\infty,
\end{align*}
then the process $u=(u(x))_{x\in\bR^m}$ is stochastically continuous.
\end{itemize}
\end{proposition}
\begin{proof}
i) This is a direct consequence of [\ref{Berger}, Proposition 5.2].\\
ii) By [\ref{Rajput}, Theorem 2.7] we have to show that
\begin{align}
&\label{eas1}\int\limits_{\R^d} \big|\gamma \big(E(x_n,y)-E(x,y)\big) + \int\limits_{\R} r \big(E(x_n,y)-E(x,y)\big)\left(\mathds{1}_{|r(E(x_n,y)-E(x,y))|\le 1}-\mathds{1}_{|r|\le 1}\right)\nu(dr)\big|\lambda^d(dy)\to 0,\\
 &\label{eas2}\int\limits_{\R^d}  \int\limits_{\R} \min \big(1, \big|r \big(E(x_n,y)-E(x,y)\big)\big|^2\big)\nu(dr)\lambda^d(dy)\to 0\quad \textrm{ and }\\
a^2&\label{eas3}\int\limits_{\R^d}\big| \big(E(x_n,y)-E(x,y)\big) \big|^2\lambda^d(dy)\to 0
\end{align}
as $n\to\infty$, if $x_n\to x$ for $n\to \infty$. At first we observe that
\begin{align*}
    \int\limits_{\R^m} \big|\gamma \big(E(x_n,y)-E(x,y)\big)|\lambda^m(dy) \le |\gamma| \|E(x_n,\cdot)-E(x,\cdot)\|_{L^1(\R^d)}\to 0
\end{align*}
as $n\to\infty$. With similar calculations as in the proof of [\ref{Berger}, Theorem 3.4] we can estimate the remaining term in (\ref{eas1}) by
\begin{align*}
    &\int\limits_{\R^d}  \int\limits_{\R} \big|r \big(E(x_n,y)-E(x,y)\big)\left(\mathds{1}_{|r(E(x_n,y)-E(x,y))|\le 1}-\mathds{1}_{|r|\le 1}\right)\big|\nu(dr)\lambda^d(dy)\\
    =& \int\limits_{\R} |r|\mathds{1}_{|r|\le1} \int\limits_{\R^d} | E(x_n,y)-E(x,y)|\mathds{1}_{| E(x_n,y)-E(x,y)|> \frac{1}{|r|}}\lambda^d(dy)\nu(dr)\\
    +&\int\limits_{\R} |r|\mathds{1}_{|r|>1} \int\limits_{\R^d} | E(x_n,y)-E(x,y)|\mathds{1}_{| E(x_n,y)-E(x,y)|\le \frac{1}{|r|}}\lambda^d(dy)\nu(dr).
\end{align*}
As it holds
\begin{align*}
    \int\limits_{\R^d} | E(x_n,y)-E(x,y)|\mathds{1}_{| E(x_n,y)-E(x,y)|> \frac{1}{|r|}}\lambda^d(dy) \le |r| \|E(x_n,\cdot)-E(x,\cdot)\|_{L^2(\R^d)}^2,
\end{align*}
it follows from Lebesgue's dominated convergence theorem that
\begin{align*}
     \int\limits_{\R} |r|\mathds{1}_{|r|\le1} \int\limits_{\R^m} | E(x_n,y)-E(x,y)|\mathds{1}_{| E(x_n,y)-E(x,y)|> \frac{1}{|r|}}\lambda^m(dy)\nu(dr)\to 0
\end{align*}
as $n\to\infty$. For the last term in (\ref{eas1}) we observe by [\ref{Grafakos}, Prop. 1.13 and 1.14] that
\begin{align*}
     &\int\limits_{\R^d} | E(x_n,y)-E(x,y)|\mathds{1}_{| E(x_n,y)-E(x,y)|\le \frac{1}{|r|}}\lambda^d(dy)\\
     \le& \int\limits_{0}^{1/|r|} d_{|E(x_n,\cdot)-E(x,\cdot)|}(\alpha) \lambda^1(d \alpha)\\
     \le&\int\limits_{0}^{1/|r|} d_{|E(x_n,\cdot)|}(\alpha/2) \lambda^1(d \alpha)+\int\limits_{0}^{1/|r|} d_{|E(x,\cdot)|}(\alpha/2) \lambda^1(d \alpha)\\
     \le&2 \left(\int\limits_{0}^{\frac{1}{2|r|}} d_{|E(x_n,\cdot)|}(\alpha) \lambda^1(d \alpha)+\int\limits_{0}^{\frac{1}{2|r|}} d_{|E(x,\cdot)|}(\alpha) \lambda^1(d \alpha)\right).
\end{align*}
By Lebesgue's dominated convergence theorem we obtain that
\begin{align*}
    \int\limits_{\R} |r|\mathds{1}_{|r|>1} \int\limits_{\R^d} | E(x_n,y)-E(x,y)|\mathds{1}_{| E(x_n,y)-E(x,y)|\le \frac{1}{|r|}}\lambda^d(dy)\nu(dr)\to 0\textrm{ as }n\to\infty.
\end{align*}
So we showed (\ref{eas1}). In order to see (\ref{eas2}) observe that
\begin{align*}
    \int\limits_{\bR^d} \mathds{1}_{|r (E(x_n,y)-E(x,y))|>1}\lambda^d(dy)\le |r|\|E(x_n,\cdot)-E(x,\cdot)\|_{L^1(\bR^d)}\to 0\textrm{ as }n\to\infty
\end{align*}
and
\begin{align*}
    \int\limits_{\bR^d} \mathds{1}_{|r (E(x_n,y)-E(x,y))|>1}\lambda^d(dy)
    \le d_{|E(x_n,\cdot)|}\left(\frac{1}{2|r|}\right)+d_{|E(x,\cdot)|}\left(\frac{1}{2|r|}\right).
\end{align*}
Now by similar arguments as in the proof of [\ref{Berger}, Theorem 3.4] we see that (\ref{eas2}) holds true. Furthermore, it is clear that (\ref{eas3}) holds, since $T_E$ is continuous.
\end{proof}

Now we state under which conditions a mild solution of a stochastic partial differential equation gives rise to a generalized solution. 

\begin{theorem}\label{mildGen}
Let $\dot{L}$ be a L\'evy white noise on $\R^d$ with characteristic triplet $(a,\gamma,\nu)$ with existing first moment and $p(x,D)$ be a partial differential operator of the form
\begin{align*}
    p(x,D)\varphi(x)=- \emph{div} (A\nabla \varphi(x))+b(x)\cdot\nabla \varphi(x)+V(x)\varphi(x),
\end{align*}
where $b\in C^1(\R^d,\R^d)$ and $V\in L^1_{loc}(\R^d)$
such that there exists a weak fundamental solution $E:\bR^d\times\bR^d\to \bR$ of the equation $p(x,D)u =\delta_0$ with $E(x,\cdot)\in L^1(\bR^d)\cap L^2(\bR^d)\cap D(\dot{L})$ for all $x\in\bR^d$ and 
\begin{align*}
    \int\limits_{K} \|E(x,\cdot)\|_{L^p(\bR^d)}^p \lambda^d(dx)<\infty
\end{align*}
for all compact sets $K\subset \bR^d$ for $p=1,2$. Then the mild solution
\begin{align*}
    u(x) = \langle \dot{L}, E(x,\cdot) \rangle
\end{align*}
of $p(x,D) u =  \dot{L}$ gives rise to a generalized solution $s$ of the stochastic partial differential equation $p(x,D) s = \dot{L}$ via
\begin{align*}
    \langle s, \varphi \rangle := \int\limits_{\R^d} u(x)\varphi(x) \lambda^d(dx), \quad \varphi\in\mathcal{D}(\R^d).
\end{align*}
\end{theorem}

\begin{proof}
We want to apply a stochastic Fubini theorem. Therefore we have to show that
\begin{align}\label{eq669}
    \int\limits_{\R^d}\int\limits_{\R^d}\int\limits_{\R} \min\left( |r E(x,y)\varphi(x)|,|r E(x,y)\varphi(x)|^2  \right)\nu(dr)  \lambda^d(dy) \lambda^d(dx) <\infty.
\end{align}
With similar calculations as done in the proof of [\ref{Berger}, Proposition 5.6] we get that for every $\varphi\in\mathcal{D}(\R^d)$
\begin{align*}
    \min\left( |r E(x,y)\varphi(x)|,|r E(x,y)\varphi(x)|^2  \right)\le \mathds{1}_{|r|>1} |r E(x,y)\varphi(x)| + \mathds{1}_{|r|\le 1} |r E(x,y)\varphi(x)|^2.
\end{align*}
Let $\varphi \in \mathcal{D}(\bR^d)$ such that $\supp \varphi\subset B_R(0)$, $R>0$. We observe that 
\begin{align*}
    &\int\limits_{\R^d}\int\limits_{\R^d}\int\limits_{\R} \mathds{1}_{|r|>1} |r E(x,y)\varphi(x)| \nu(dr)  \lambda^d(dy) \lambda^d(dx)\\
    \le& \| \varphi \|_{\infty} \int\limits_{\R} \mathds{1}_{|r|>1} |r| \nu(dr)   \int\limits_{B_R(0)} \| E(x,\cdot) \|_{L^1(\bR^d)} \lambda^d(dx)<\infty
\intertext{ and }
    & \int\limits_{\R^d}\int\limits_{\R^d}\int\limits_{\R} \mathds{1}_{|r|\le 1} |r E(x,y)\varphi(x)|^2 \nu(dr)  \lambda^d(dy) \lambda^d(dx)\\
    \le&  \|\varphi\|^2\int\limits_{\R} \mathds{1}_{|r|\le 1} |r|^2 \nu(dr)   \int\limits_{B_R(0)} \| E(x,\cdot) \|_{L^2(\bR^d)}^2 \lambda^d(dx)<\infty.
\end{align*}
This shows \eqref{eq669}. Since $\varphi\in\mathcal{D}(\R^d)$ has compact support and we have that $\lambda^d$ is finite on the support of $\varphi$ and with [\ref{Barndorff}, Theorem 3.1 p. 926] we get that
\begin{align*}
    \langle s,\varphi \rangle :&= \int\limits_{\R^d} u(x) \varphi(x) \lambda^d(dx) = \int\limits_{\R^d} \int\limits_{\R^d} E(x,y) \varphi(x) dL(y) \lambda^d(dx) \\
    &= \int\limits_{\R^d}\int\limits_{\R^d} E(x,y) \varphi(x) \lambda^d(dx) dL(y) \quad \textrm{ a.s.}
\end{align*}
and further it can be chosen a version of $u$ such that $u(t)\varphi(t)$ is integrable with respect to $\lambda^d$. The linearity of $s:\mathcal{D}(\R^d)\to L^0(\Omega)$ is clear and the estimates above show that it is also continuous, hence $s$ is a generalized random process. In order to see that $p(x,D)s =\dot{L}$, we observe that for arbitrary $f\in\mathcal{D}(\bR^d)$
\begin{align*}
&\int\limits_{\R^d} \left(\,\,\int\limits_{\R^d} E(x,y) p(x,D)^* \varphi(x) \lambda^d(dx)\right) f(y) \lambda^d(dy)\\
=&
    -\int\limits_{\R^d}\left(\,\,\int\limits_{\R^d} E(x,y) f(y) \lambda^d(dy)\right) (\textrm{div}(A^T(x)\nabla \varphi(x)) \lambda^d(dx)\\
    &-\int\limits_{\R^d}\left(\,\,\int\limits_{\R^d} E(x,y) f(y) \lambda^d(dy)\right) \nabla\cdot (b(x) \varphi(x)) \lambda^d(dx)\\
    &+\int\limits_{\R^d}\left(\,\,\int\limits_{\R^d} E(x,y) f(y) \lambda^d(dy)\right) V(x) \varphi(x) \lambda^d(dx)\\
    =&   \int\limits_{\R^d} \langle A(x)\nabla \left(\,\, \int\limits_{\R^d} E(x,y) f(y) \lambda^d(dy)\right), \nabla\varphi(x)\rangle\lambda^d(dx) \\
    &+\int\limits_{\R^d}  (b(x)\cdot\nabla+V(x)) \left(\,\, \int\limits_{\R^d} E(x,y) f(y) \lambda^d(dy)\right) \varphi(x)\lambda^d(dx)
    =& \int\limits_{\R^d} f(x)\varphi(x)\lambda^d(dx).
\end{align*}
As $f\in \mathcal{D}(\bR^d)$ was arbitrary, it follows from the fundamental lemma of calculus of variations that
\begin{align*}
    \int\limits_{\R^d} E(x,y) p(x,D)^* \varphi(x) \lambda^d(dx)=\varphi(y)\textrm{ a.e.}
\end{align*}
Now we obtain
\begin{align*}
    \langle s,p(x,D)^*\varphi\rangle &= \int\limits_{\R^d} \int\limits_{\R^d} E(x,y) p(x,D)^* \varphi(x) dL(y) \lambda^d(dx)= \int\limits_{\R^d} \varphi(y)  dL(y) = \langle \dot{L}, \varphi \rangle,
\end{align*}
so we see that $s$ is a generalized solution.
\end{proof}

\subsection{Proof of Theorem \ref{SchroderSol}}\label{secproof}
\begin{proof}
i) Similar to [\ref{Shen}, Remark 3.21] we observe that we can estimate the distance function $\gamma$ in (\ref{mayborodaeq}) and obtain for the weak fundamental solution $E$ of the generalized electric Schr\"odinger operator $p(x,D)$ that it holds
\begin{align}
    |E(x,y)| \le C_1 e^{-C_2 (1+m(x,V)\|x-y\|)^{\theta}}\label{sheneq} \|x-y\|^{2-d},
\end{align}
for some constants $C_1, C_2>0$ and $0<\theta<1$. Hence, we obtain that
\begin{align*}
    \int\limits_{\R^d} |E(x,y)| \lambda^d(dy) &\le C_1 \int\limits_{\R^d} e^{-C_2(1+m(x,V) \| z\|)^{\theta}} \|z\|^{2-d} \lambda^d(dz) \\
    &= C_1\int\limits_{0}^{\infty} r e^{-C_2(1+m(x,V)  r)^{\theta}} \lambda^1(dr)  < \infty
\intertext{and also}
    \int\limits_{\R^d} |E(x,y)|^2 \lambda^d(dy) &\le  C_3 \int\limits_{0}^{\infty} r^{3-d} e^{-2C_2(1+m(x,V)  r)^{\theta}} \lambda^1(dr) < \infty,
\end{align*}
where $C_3>0$. For $\alpha>0$ and $x,y\in\R^d$, $x\neq y$ it follows with the triangular inequality that (observe that $d=3$ and hence the Lebesgue measure of a ball with radius $r$ is $\frac{4\pi}{3}$)
\begin{align*}
    &\lambda^d\left( \{ y\in\R^d: \frac{Ce^{-k\gamma(x,y,V)}}{\|x-y\|^{d-2}}> \alpha \} \right) \\
    \le &  \lambda^d\left( \{ y\in\R^d \setminus B_{e^{-k\gamma(x,0,V)/(d-2)}C^{1/(d-2)}}(x) :  \frac{e^{-k \gamma(0,y,V)}}{||x-y||^{d-2}}> e^{k\gamma(x,0,V)}\alpha/C \} \right) + \frac{4\pi}{3}(e^{-k\gamma(x,0,V)/(d-2)}C^{1/(d-2)})^d \\
    \le &  \lambda^d\left( \{ y\in\R^d\setminus B_{e^{-k\gamma(x,0,V)/(d-2)}C^{1/(d-2)}}(x): d(0,y,V)\le - \frac{1}{k} \log(\alpha )\right) + \frac{4\pi}{3} (e^{-k\gamma(x,0,V)/(d-2)}C^{1/(d-2)})^d \\
    \le& \lambda^d\left(B^{V}(0,- \frac{1}{k} \log(\alpha ))\right) + \frac{4\pi}{3}  (e^{-k\gamma(x,0,V)/(d-2)}C^{1/(d-2)})^d .
\end{align*}
It follows with (\ref{mayborodaeq}) that
\begin{align*}
    &\int\limits_{\R} |r| \mathds{1}_{|r|>1} \int\limits_{0}^{1/|r|} d_{E(x,\cdot)}(\alpha) \lambda^1(d\alpha) \nu(dr) \\
    \le& C_4(x)\left(1+ \int\limits_{|r|>1} |r| \int\limits_{0}^{1/|r|} \lambda^d\left(B^{V}(0,- \frac{1}{k} \log(\alpha))\right)\lambda^1(d\alpha) \nu(dr)\right)<\infty
\end{align*}
by assumption, where $0<C_4(x)<\infty$. Proposition \ref{mild} i) now gives the existence of a mild solution.\\ 
To show the continuity of the mild solution by the previous estimates and Proposition \ref{mild} ii) it is sufficient to prove that $T_E:\bR^d\to L^1(\bR^d)\cap L^2(\bR^d)$, $T_E(x)(\cdot)=E(x,\cdot)$, is continuous. Let $x_0\in\bR^d$ and $(x_n)_{n\in\bN}$ be a sequence such that $x_n\to x_0$ as $n\to\infty$. Let $0<2\|x_0-x_n\|<r_0$ for all $n\ge M$, $M\in\bN$. We calculate that
\begin{align*}
\|E(x_0,\cdot)-E(x_n,\cdot)\|_{L^1(\bR^d)}\le \|E(x_0,\cdot)-E(x_n,\cdot)\|_{L^1(B_{r_0}(x))}+\|E(x_0,\cdot)-E(x_n,\cdot)\|_{L^1(\bR^d\setminus B_{r_0}(x))}.
\end{align*}
It was shown in [\ref{Mayboroda}, Lemma 3.12, page 14] that it holds for a constant $0< \kappa<1$
\begin{align}
	m(x,V) \ge C \frac{m(0,V)}{(1+\| x\|m(0,V))^{\kappa}}\textrm{ for all }x\in\bR^d \label{maxfunction},
\end{align}
hence there exists an $\varepsilon>0$ such that it follows with (\ref{sheneq}) that
\begin{align*}
|E(x_n,y)|\le C_1 e^{-C_2(1+\varepsilon\|x_n-y\|)^\theta}\|x_n-y\|^{2-d}
\end{align*}
for every $n\in \bN_0$. Therefore, we obtain that
\begin{align*}
   \|E(x_0,\cdot)-E(x_n,\cdot)\|_{L^1(B_{r_0}(x))}\le 2\int_{B_{2r_0}(0)}  C_1 e^{-C_2(1+\varepsilon\|y\|)^\theta}\|y\|^{2-d}\lambda^d(dy).
\end{align*}
and
\begin{align*}
    |E(x_0,y)-E(x_n,y)|\le C_1e^{-C_2(1+\varepsilon\|x_n-y\|)^\theta}\|x_n-y\|^{2-d}+C_1e^{-C_2(1+\varepsilon\|x_0-y\|)^\theta}\|x_0-y\|^{2-d}.
\end{align*}
As $(x_n)_{n\ge M}$ is bounded we  can find an integrable majorant on $\bR^d\setminus B_{r_0}(x)$. We know from [\ref{Mayboroda}, chapter 7] that $E$ is continuous and by Lebesgue's Dominated Convergence Theorem we obtain
\begin{align*}
\lim_{n\to\infty}\|E(x_0,\cdot)-E(x_n,\cdot)\|_{L^1(\bR^d\setminus B_{r_0}(x))}=0.
\end{align*}
We see that
\begin{align*}
    \lim_{n\to\infty} \|E(x_0,\cdot)-E(x_n,\cdot)\|_{L^1(\bR^d)}\le 2\int_{B_{2r_0}(0)}  C_1 e^{-C_2(1+\varepsilon\|y\|)^\theta}\|y\|^{2-d}\lambda^d(dy).
\end{align*}
By letting $r_0$ go to $0$ we obtain that $\lim_{n\to\infty} \|E(x_0,\cdot)-E(x_n,\cdot)\|_{L^1(\bR^d)}=0$.
The same proof works for the $L^2$-norm.\\
ii) Let $\tilde{E}$ be the left inverse of $p(x,D)^*$, i.e it holds 
\begin{align*}
    \int\limits_{\R^d} \tilde{E}(x,y) p(y,D)^* \varphi(y) \lambda^d(dy)=\varphi(x)
\end{align*}
for $\varphi\in\mathcal{D}(\R^d)$. We have to show that $\tilde{E}_R\in L^1(\R^d)\cap L^2(\R^d)$ in order to satisfy the assumptions of Theorem \ref{theo1}. As $\tilde{E}(x,y)=E(y,x)$ we can show by a similar argument as in i) that for $R>0$
\begin{align*}
    \tilde{E}_R(x)= \int\limits_{B_R(0)} |\tilde{E}(x,y)| \lambda^d(dy) \le \int\limits_{B_R(0)}C_1 e^{-C_2 (1+m(y,V)\|x-y\|)^{\theta}} \|x-y\|^{2-d}\lambda^d(dy).
\end{align*}
By using (\ref{maxfunction}) we obtain that
\begin{align*}
    \tilde{E}_R(x) &\le C_R  \int\limits_{B_R(0)}e^{-k C^1_R\|x-y\|^{\theta} }\|x-y\|^{2-d}\lambda^d(dy)\le \tilde{C}_Re^{-k C^1_R\|x\|^{\theta} }\|x\|^{2-d},
\end{align*}
where $C_R, C_R^1,\tilde{C}_R>0$. Therefore we obtain that $\| \tilde{E}_R \|_{L^1(\R^d)}+ \| \tilde{E}_R \|_{L^2(\R^d)}<\infty$. We observe from (\ref{mayborodaeq}) and [\ref{Shen}, Remark 3.21] by applying the triangular inequality that
\begin{align*}
    \tilde{E}_R(x) &\le e^{-k\gamma(x,0,V)} \int\limits_{B_R(0)} \frac{e^{k\gamma(y,0,V)}}{\|x-y\|^{d-2}} \lambda^d(dy)\le  C'_R e^{-k\gamma(x,0,V)} \int\limits_{B_{R}(x)} \frac{1}{\|y\|^{d-2}}\lambda^d(dy)\le C''_{R} \frac{e^{-k\gamma(x,0,V)}}{\|x\|^{d-2}}, 
\end{align*}
where $C'_R,C''_{R}>0$ are constants dependent on $R$. This leads with similar arguments as in i) to
\begin{align*}
   \int\limits_{\R} |r| \mathds{1}_{|r|>1} \int\limits_0^{1/|r|} d_{\tilde{E}_R}(\alpha) \lambda^1(d\alpha)\nu(dr) \le C_R\left(1+ \int\limits_{|r|>1} |r| \int\limits_{0}^{1/|r|} \lambda^d\left(B^{V}(0,- \frac{1}{k} \log(\alpha))\right) \lambda^1(d\alpha) \nu(dr)\right)<\infty,
\end{align*}
for a constant $C_R>0$ dependent on $R>0$. With Theorem \ref{theo1} follows the existence of a generalized solution $s:\mathcal{D}(\R^d)\to L^0(\Omega)$.\\
iii) Given the mild solution from i) we obtain with (\ref{maxfunction}) for $R>0$ that
\begin{align*}
    \int\limits_{B_R(0)} \| E(x,\cdot)\|_{L^1(\R^d)} \lambda^d(dx) <\infty
\end{align*}
and
\begin{align*}
    \int\limits_{B_R(0)} \| E(x,\cdot)\|_{L^2(\R^d)} \lambda^d(dx)<\infty.
\end{align*}
Hence, we obtain the assertion by Theorem \ref{mildGen}.
\end{proof}
\section*{Acknowledgement:}  The first author is financially supported through the DFG-NCN Beethoven Classic 3 project SCHI419/11-1. The two authors would like to thank Alexander Lindner for his support and for many interesting and fruitful discussions. Moreover, the authors would like to thank Ren\'e Schilling for his comments, which helped to improve the paper greatly.

\mbox{}\\
David Berger\\
 TU Dresden, Institute of Mathematical Stochastics, Zellescher Weg 12-14, 01069 Dresden,
Germany\\
email: david.berger2@tu-dresden.de\\\\
Farid Mohamed\\
Ulm University, Institute of Mathematical Finance,  Helmholtzstra{\ss}e 18, 89081 Ulm,
Germany\\
email: farid.mohamed@uni-ulm.de


\begin{thebibliography}{LPSt}

\bibitem{Barndorff} O.E.Barndorff-Nielsen and A.Basse-O'Connor, \textit{Quasi Ornstein-Uhlenbeck processes}, Bernoulli (2011) 17, no. 3, 916-941.\label{Barndorff}

\bibitem{Berger} D. Berger, \textit{L\'evy driven CARMA generalized processes and stochastic partial differential equations}, accepted in Stochastic Processes and their Applications, (2020). \label{Berger}

\bibitem{Davey} B. Davey, J. Hill and S. Mayboroda, \textit{Fundamental matrices and Green matrices for non-homogeneous elliptic systems}, Publ. Mat. 62 (2018), no. 2, 537--614.\label{Davey}

\bibitem{Dalang} R. C. Dalang and T. Humeau, \textit{Random field solutions to linear SPDEs driven by symmetric pure jump L\'{e}vy space-time white noise}, Electron. J. Probab. 24 (2019), no. 6, 1--28.\label{Dalang}

\bibitem{Fageot2} J. Fageot, A. Amini, and M. Unser, \textit{On the continuity of characteristic functionals and sparse stochastic modeling}, Journal of Fourier Analysis and Applications, 20:1179–
1211, 2014.\label{Fageot2}

\bibitem{Fageot} J. Fageot and T. Humeau, \textit{Unified View on L\'evy White Noises: General Integrability Conditions and Applications to Linear SPDE}, (2018) arXiv:1708.02500. \label{Fageot}

\bibitem{Gelfand} I.M. Gelfand and N.Y. Vilenkin, \textit{Generalized Functions, Vol. 4: Applications of Harmonic Analysis}, Academic Press, New York and London, (1964). \label{Gelfand}

\bibitem{Grafakos} L. Grafakos, \textit{Classical Fourier Analysis}, Second edition, Springer, (2008). \label{Grafakos}

\bibitem{Holden} H. Holden, H. {\O}ksendal, B. Ub{\o}e and T. Zhang, \textit{ Stochastic Partial Differential Equations: A Modelling White Noise Functional Approach}, Springer-Verlag New York (2010). \label{Holden}
\bibitem{Hoermander} L. H\"ormander, \textit{The Analysis of Linear Partial Differential Operators I: Distribution Theory and Fourier Analysis}, Springer-Verlag Berlin Heidelberg (2003).\label{Hoermander} 

\bibitem{Littman} W. Littman, G. Stampacchia and H.F. Weinberger, \textit{Regular points for elliptic equations with
discontinuous coefficients}, Ann. Sc. Norm. Sup. Pisa 17 (1963), 43–77. \label{Littman}

\bibitem{Lokka} A. L{\o}kka, B. {\O}ksendal, F. Proske, \textit{Stochastic partial differential equations driven by L\'evy space-time white noise}, The Annals of Applied Probability 14 (3) (2004), 1506-1528. \label{Lokka} 

\bibitem{Mayboroda} S. Mayboroda and B. Poggi \textit{Exponential decay estimates for fundamental solutions of Schr{\"o}dinger-type operators}, Trans. Amer. Math. Soc. 372 (2019), 4313-4357. \label{Mayboroda}

\bibitem{Rajput} B. S. Rajput and J. Rosinski, \textit{Spectral Representations of Infinitely Divisible Processes}, Probab. Th. Rel. Fields, (1989) 82, 451-487.  \label{Rajput}

\bibitem{Sato} K. Sato, \textit{L\'evy Processes and Infinitely Divisible Distributions}, Cambridge University Press, Cambridge, 2013.\label{Sato}

\bibitem{Stein} E. Stein, \textit{Harmonic Analysis}, Princeton University Press, 1993. \label{Stein}

\bibitem{Shen} Z. Shen, \textit{On fundamental solutions of generalized Schr\"odinger operators}, J. Funct. Anal. 167 (1999), 521–567. \label{Shen}

\bibitem{van Putten} M. van Putten, \textit{Maxwell's equations in divergence form for general media with applications to MHD}, Commun.Math. Phys. 141, 63–77 (1991).\label{Putten}

\bibitem{Walsh} J.B. Walsh, \textit{An introduction to stochastic partial differential equations},  \'{E}cole d'\'{E}t\'{e} de Probabilit\'{e}s de Saint Flour (1986) XIV-1984, pages 265-439. Springer.\label{Walsh}


\end{thebibliography}
\end{document}